\documentclass[11pt]{amsart}
\usepackage{amsmath,amsfonts,amssymb,amsthm,amscd}
\usepackage{graphicx}
\usepackage[matrix,arrow,curve]{xy}
\begin{document}

\theoremstyle{plain}
\newtheorem{thm}{Theorem}
\newtheorem{lmm}{Lemma}
\newtheorem{claim}{Claim}
\newtheorem{prop}{Proposition}
\newtheorem{coro}{Corollary}
\newtheorem*{tm}{Theorem}
\newtheorem*{lm}{Lemma}
\newtheorem*{clm}{Claim}
\newtheorem*{prp}{Proposition}
\newtheorem*{cor}{Corollary}

\theoremstyle{definition}
\newtheorem{dfn}{Definition}
\newtheorem*{df}{Definition}
\newtheorem*{nt}{Notation}
\newtheorem*{obs}{Observation}
\newtheorem*{exm}{Example}
\newtheorem*{exms}{Examples}
\newtheorem{prbl}{Problem}
\newtheorem*{prb}{Problem}

\theoremstyle{remark}
\newtheorem*{pf}{Proof}
\newtheorem*{rmk}{Remark}
\newtheorem*{q}{Question}
\newtheorem*{qs}{Questions}
\newtheorem*{ex}{Example}
\newtheorem*{exs}{Examples}
\newtheorem*{ack}{Acknowledgements}
\newtheorem*{fact}{Fact}
\newtheorem*{facts}{Facts}

\newcommand{\ob}{[}
\newcommand{\cb}{]}

\newcommand{\dom}{ {\mathop{\mathrm {dom\,}}\nolimits} }
\newcommand{\ran}{ {\mathop{\mathrm{ran\,}}\nolimits} }
\newcommand{\Ran}{ {\mathop{\mathrm{Range\,}}\nolimits} }

\newcommand{\cf}{ {\mathop{\mathrm {cf\,}}\nolimits} }
\newcommand{\inter}{ {\mathop{\mathrm {int\,}}\nolimits} }
\newcommand{\cl}{ {\mathop{\mathrm {cl\,}}\nolimits} }
\newcommand{\cll}{ {\mathop{\mathrm {cl^{-}\,}}\nolimits} }
\newcommand{\clu}{ {\mathop{\mathrm {cl^{+}\,}}\nolimits} }
\newcommand{\lcl}{ {\mathop{\mathrm {lcl\,}}\nolimits} }
\newcommand{\rcl}{ {\mathop{\mathrm {rcl\,}}\nolimits} }
\newcommand{\cof}{ {\mathop{\mathrm {cof\,}}\nolimits} }
\newcommand{\add}{ {\mathop{\mathrm {add\,}}\nolimits} }
\newcommand{\sat}{ {\mathop{\mathrm {sat\,}}\nolimits} }
\newcommand{\tc}{ {\mathop{\mathrm {tc\,}}\nolimits} }
\newcommand{\unif}{ {\mathop{\mathrm {unif\,}}\nolimits} }
\newcommand{\uhr}{\!\upharpoonright\!}
\newcommand{\lra}{ {\leftrightarrow} }
\newcommand{\ot}{ {\mathop{\mathrm {ot\,}}\nolimits} }
\newcommand{\pr}{ {\mathop{\mathrm {pr\,}}\nolimits} }
\newcommand{\cnc}{ {^\frown} }
\newcommand{\image}{\/``\,}
\newcommand{\scc}{\beta\!\!\!\!\beta}
\newcommand{\filt}{\varepsilon\!\!\!\!\!\;\varepsilon}
\newcommand{\pws}{P\!\!\!\!\!P}
\newcommand{\ol}{\overline}
\newcommand{\wh}{\widehat}
\newcommand{\wt}{\widetilde}

\newcommand{\supp}{\mathrm{supp}}
\newcommand{\defin}{\mathrm{def}}
\newcommand{\Pos}{\mathrm{Pos}}
\newcommand{\id}{\mathrm{id}}

\pagestyle{headings}

  \title{Ultrafilter extensions do not preserve 
         elementary equivalence}

  \author{Denis I.~Saveliev, Saharon Shelah}
  
  \date{4 May 2013, the last revision 03 Nov 2019}


\thanks{
{\it Mathematical Subject Classification 2010\/}:
Primary 
03C55, 
54D35, 
54D80, 
Secondary 
03C30, 
03C80, 
54H10} 

\thanks{
{\it Keywords\/}:
ultrafilter, 
ultrafilter extension, 
\v{C}ech--Stone compactification, 
first-order model,
elementary equivalence,
elementary embedding, 
ultrafilter quantifier}

\thanks{
{\it Acknowledgment\/}: 
This research was partially supported by 
European Research Council grant~338821. 
The first author was also partially supported by 
Russian Foundation for Basic Research grant~17-01-00705. 
}

\thanks{
Paper~1132 on Shelah's list.}

\begin{abstract}
We show that there exist models $\mathcal M_1$ 
and $\mathcal M_2$ such that $\mathcal M_1$ elementarily 
embeds into~$\mathcal M_2$ but their ultrafilter 
extensions $\scc(\mathcal M_1)$ and $\scc(\mathcal M_2)$ 
are not elementarily equivalent.
\end{abstract}

\maketitle

\section{Introduction}

The ultrafilter extension of a~first-order model is 
a~model in the same vocabulary, the universe of which 
consists of all ultrafilters on the universe of the 
original model, and which extends the latter in 
a~canonical way. This construction was introduced 
in~\cite{Saveliev}. The article~\cite{Saveliev(inftyproc)} 
is an expanded version of~\cite{Saveliev}; it contains 
a~list of problems, one of which is solved here. 

The main precursor of the general construction was 
the ultrafilter extension of semigroups, called often 
the \v{C}ech--Stone compactification of semigroups. 
This particular case was discovered in 1970s and became 
since then an important tool for getting various 
Ramsey-theoretic results in combinatorics, algebra, 
and dynamics; the textbook~\cite{Hindman Strauss} is 
a~comprehensive treatise of this area. For theory of 
ultrafilters and for model theory we refer the reader 
to the standard textbooks \cite{Comfort Negrepontis} 
and~\cite{Chang Keisler}, respectively.


Recall the construction of ultrafilter extensions
and related basic facts.

\begin{dfn}\label{ultrafilter quantifier} 
For a~set~$M$, an ultrafilter~$D$ on~$M$, and a~formula 
$\varphi(x,\ldots)$ with parameters $x,\ldots$\,, we let
\begin{align*}
(\forall^{D}x)\,\varphi(x,\ldots)
&\;\;\text{if and only if}\;\;
\{a\in M:\varphi(a,\ldots)\}\in D.
\end{align*}
\end{dfn}

It is easy to see that the ultrafilter quantifier is 
self-dual: it coincides with $(\exists^Dx)$, defined 
as $\neg\,(\forall^Dx)\,\neg\,$, since $D$~is ultra.
Note also that if $D$~is the principal ultrafilter given 
by some $a\in M$, then $(\forall^Dx)\,\varphi(x,\ldots)$ 
is reduced to $\varphi(a,\ldots)$, and 
that, e.g.,
$
(\forall^{D_1}x_1)(\forall^{D_2}x_2)\,
\varphi(x_1,x_2,\ldots)
$
means 
$
\{a_1\in M:\{a_2\in M:
\varphi(a_1,a_2,\ldots)\}\in D_2\}\in D_1.
$

\begin{dfn}\label{ultrafilter extension}
Let $\mathcal M$~be a~model in a~vocabulary~$\tau$ with 
the universe~$M$. Define the model~$\scc(\mathcal M)$ and 
the function~$j_M$ as follows:
\begin{itemize}
\item[(a)]
the universe of~$\scc(\mathcal M)$ is $\scc(M)$, 
the set of ultrafilters on~$M$,
\item[(b)]
$j_M:M\to\scc(M)$ is such that for all $a\in M$, 
$j_M(a)$~is the principal ultrafilter on~$M$ given by~$a$,
i.e., $j_M(a)=\{A\subseteq M:a\in A\}$,
\item[(c)]
if $P\in\tau$ is an $n$-ary predicate symbol
(other than the equality symbol), let
$$
P^{\scc(\mathcal M)}=\bigl\{(D_1,\ldots,D_{n}):
(\forall^{D_1}x_1)\ldots(\forall^{D_{n}}x_{n})\,
P^{\mathcal M}(x_1,\ldots,x_{n})\bigr\},
$$
\item[(d)]
if $F\in\tau$ is an $n$-ary function symbol, let 
\begin{gather*}
F^{\scc(\mathcal M)}(D_1,\ldots,D_{n})=D
\;\;\text{if and only if}
\\
\bigl(\forall A\subseteq M)\,
\bigl(A\in D
\;\Leftrightarrow\;
(\forall^{D_1}x_1)\ldots(\forall^{D_{n}}x_{n})\,
F^{\mathcal M}(x_1,\ldots,x_{n})\in A\bigr).
\end{gather*}
\end{itemize}
The model $\scc(\mathcal M)$ is the 
{\it ultrafilter extension\/} of the model~$\mathcal M$, 
and $j_M$~is the {\it natural embedding\/}
of $\mathcal M$ into~$\scc(\mathcal M)$.
\end{dfn}

The using of words ``extension'' and ``embedding'' 
is easily justified:

\begin{prop}\label{embedding into ultrafilter extension}
If $\mathcal M$~is a~model in a~vocabulary~$\tau$, then 
\begin{itemize}
\item[(a)]
$\scc(\mathcal M)$~is also a~model in~$\tau$, and
\item[(b)]
$j_M$~isomorphically embeds $\mathcal M$ into~$\scc(\mathcal M)$.
\end{itemize}
\end{prop}

\begin{proof}
See~\cite{Saveliev},~\cite{Saveliev(inftyproc)}.
\end{proof}


The following result, called the First Extension Theorem 
in~\cite{Saveliev(inftyproc)}, shows that the ultrafilter 
extension lifts certain relationships between models.

\begin{thm}\label{first extension theorem}
Let $\mathcal M_1$ and $\mathcal M_2$ be two models in the same 
vocabulary with the universes $M_1$ and $M_2$, respectively,
and let $h$~be a~mapping of $M_1$ into~$M_2$ and $\wt h$~its 
(unique) continuous extension of $\scc(M_1)$ into $\scc(M_2)$:
$$
\xymatrix{
\;\scc(\mathcal M_1)\ar@{-->}^{\wt h}[r]&\;\scc(\mathcal M_2)
\\
\mathcal M_1\ar[r]^{h}\ar[u]^{j_{M_1}}&\mathcal M_2\ar[u]^{j_{M_2}}
}
$$
If $h$~is a~homomorphism (epimorphism, isomorphic embedding)
of $\mathcal M_1$ into $\mathcal M_2$, then $\wt h$~is 
a~homomorphism (epimorphism, isomorphic embedding) 
of $\scc(\mathcal M_1)$ into $\scc(\mathcal M_2)$.
\end{thm}

\begin{proof}
See~\cite{Saveliev},~\cite{Saveliev(inftyproc)}.
\end{proof}

Actually Theorem~\ref{first extension theorem} is 
a~special case of a~stronger result, called the 
Second Extension Theorem in~\cite{Saveliev(inftyproc)}. 
Here we omit its precise formulation, which involves
topological concepts, and note only that it generalizes 
the standard topological fact stating that the
\v{C}ech--Stone compactification is the largest one, 
to the case when the underlying discrete space~$M$ 
carries an arbitrary first-order structure. This 
confirms that the construction of ultrafilter extensions 
given in Definition~\ref{ultrafilter extension} is 
canonical in a~certain sense.

Theorem~\ref{first extension theorem} holds also for 
certain other relationships between models (e.g., for
so-called homotopies and isotopies, see
\cite{Saveliev},~\cite{Saveliev(inftyproc)}). 
A~natural task is a~characterization of such relationships. 
In particular, one can ask whether elementary embeddings or 
elementary equivalence lift under ultrafilter extensions. 
This task was posed in~\cite{Saveliev(inftyproc)} (see 
Problem~5.1 there and comments before it).

In this note, we answer this particular question in the 
negative. In fact, we establish a~slightly stronger result:

\begin{thm}[the Main Theorem]\label{main theorem}
There exist models $\mathcal M_1$ and $\mathcal M_2$ in the 
same vocabulary such that $\mathcal M_1$ elementarily embeds 
into~$\mathcal M_2$ but their ultrafilter extensions 
$\scc(\mathcal M_1)$ and $\scc(\mathcal M_2)$ are not 
elementarily equivalent:
$$
\xymatrix{
\;\scc(\mathcal M_1)\ar@{-->}^{\not\equiv}[r]&\;\scc(\mathcal M_2)
\\
\mathcal M_1\ar[r]^{\prec}\ar[u]^{j_{M_1}}&\mathcal M_2\ar[u]^{j_{M_2}}
}
$$
\end{thm}

Of course, it follows that neither elementary embeddings 
nor elementary equivalence are preserved under ultrafilter 
extensions. The construction of such models $\mathcal M_1$ 
and~$\mathcal M_2$ will be provided in the next section. 

We conclude this section with the following natural 
questions on possible general results in this direction. 

\begin{prbl}
Characterize (or at least, provide interesting 
necessary or sufficient conditions on) theories~$T$ 
such that the implication 
$$
\mathcal M_1\equiv\mathcal M_2
\;\;\Rightarrow\;\;
\scc(\mathcal M_1)\equiv\scc(\mathcal M_2)
$$
holds for all $\mathcal M_1,\mathcal M_2\vDash T$.
\end{prbl}

\begin{prbl}
The same question for elementary embeddings.
\end{prbl}


\section{Proof of the Main Theorem}

First we define a~vocabulary~$\tau$ and construct two 
specific models $\mathcal M_1$ and $\mathcal M_2$ in~$\tau$. 
Then we shall show that these models are as required.

\begin{dfn}
Let $\tau$ be the vocabulary consisting of 
two unary predicate symbols $P_1$ and~$P_2$, 
two binary predicate symbols $R_1$ and~$R_2$, 
and one binary function symbol~$F$.
\end{dfn}

\begin{dfn}\label{model 1}
Let $\mathcal M_1$ be a~model in~$\tau$ having
the universe~$M_1$ and defined as follows:
\begin{itemize}
\item[(a)]
$M_1=\mathbb N\sqcup\mathcal P(\mathbb N)$, 
the disjoint sum of $\mathbb N$ and $\mathcal P(\mathbb N)$
(which we shall identify with their disjoint copies),
\item[(b)]
$P^{\mathcal M_1}_1=\mathbb N$,
\item[(c)]
$
P^{\mathcal M_1}_2=
\mathcal P(\mathbb N)
$,
\item[(d)]
$
R^{\mathcal M_1}_1=
\{(n,a):
n\in\mathbb N\wedge
a\in\mathcal P(\mathbb N)\wedge 
n\in a\}
$,
i.e., the intersection of the membership relation
with $\mathbb N\times\mathcal P(\mathbb N)$,
\item[(e)]
$R^{\mathcal M_1}_2$ is a~relation such that
\begin{itemize}
\item[($\alpha$)]
$R^{\mathcal M_1}_2\cap(\mathbb N\times\mathbb N)$ 
is the usual order on~$\mathbb N$,
\item[($\beta$)]
$
R^{\mathcal M_1}_2\cap
(\mathcal P(\mathbb N)\times\mathcal P(\mathbb N))$ 
is a~linear order on~$\mathcal P(\mathbb N)$ with 
no endpoints,
\item[($\gamma$)] 
if $a\in\mathbb N\Leftrightarrow b\notin\mathbb N$ 
then $R^{\mathcal M_1}_2(a,b)$ is defined 
arbitrarily (really this case will not be used),
\end{itemize}
\item[(f)]
$F^{\mathcal M_1}$~is an unordered pairing function 
mapping $\mathbb N$ into $\mathbb N$ and 
$\mathcal P(\mathbb N)$ into $\mathcal P(\mathbb N)$, 
i.e., satisfying the following conditions: 
\begin{itemize}
\item[($\alpha$)]
if 
either $a_1,b_1,a_2,b_2\in\mathbb N$ 
or $a_1,b_1,a_2,b_2\in\mathcal P(\mathbb N)$, 
then
$$
F^{\mathcal M_1}(a_1,b_1)=
F^{\mathcal M_1}(a_2,b_2)
\;\Leftrightarrow\;
\{a_1,b_1\}=\{a_2,b_2\},
$$
\item[($\beta$)]
if $a,b\in\mathbb N$ then
$F^{\mathcal M_1}(a,b)\in\mathbb N$,
\item[($\gamma$)]
if $a,b\in\mathcal P(\mathbb N)$ then
$F^{\mathcal M_1}(a,b)\in\mathcal P(\mathbb N)$, 
\item[($\delta$)]
if $a\in\mathbb N\Leftrightarrow b\notin\mathbb N$ 
then $F^{\mathcal M_1}(a,b)$ is defined 
arbitrarily (really this case will not be used).
\end{itemize}
\end{itemize}
\end{dfn}


\begin{prop}\label{model 2}
Assume $\lambda\ge2^{\aleph_0}$. Then 
there exists a~model~$\mathcal M_2$ in~$\tau$ 
such that $\mathcal M_1\prec\mathcal M_2$ and 
$|P^{\mathcal M_2}_1|=|P^{\mathcal M_2}_2|=\lambda$.
\end{prop}

\begin{proof}
Let $\mathcal M_3$ be $\lambda$-saturated 
and $\mathcal M_1\prec\mathcal M_3$. 
By the $\lambda$-saturatedness, 
for each $i\in\{1,2\}$ we have 
$|P^{\mathcal M_3}_i|\ge\lambda$, 
so we can pick  
$A_i\subseteq P^{\mathcal M_3}_i$ with $|A_i|=\lambda$.
By the downward L\"owenheim--Skolem Theorem, there exists
a~model~$\mathcal M_2$ with the universe~$M_2$ such that:
\begin{itemize}
\item[(a)]
$\mathcal M_2\prec\mathcal M_3$,
\item[(b)]
$M_1\cup A_1\cup A_2\subseteq M_2$, 
\item[(c)]
$|M_2|=\lambda$,
\end{itemize}
whence it follows that $\mathcal M_2$ is a~required model.

Alternatively, we can use a~version of the upward 
L\"owenheim--Skolem Theorem by picking two sets of 
constants, $C_1$ and~$C_2$, with $|C_1|=|C_2|=\lambda$ 
and adding to the elementary diagram of~$\mathcal M_1$ 
the formulas $P_i(c_i)$ for all $c_i\in C_i$, 
$i\in\{1,2\}$.
The obtained theory is consistent (by compactness), 
so extract its submodel of cardinality~$\lambda$ 
(by the downward L\"owenheim--Skolem Theorem) and 
reduce it to the required model~$\mathcal M_2$ 
in the original vocabulary~$\tau$. 
\end{proof}

Clearly, this observation is of a~general character;
a~similar argument allows to get, for any model, its 
elementary extension in which all predicate symbols 
are interpreted by relations of the same cardinality.

To simplify reading, we slightly shorthand the 
notation for the ultrafilter extensions of the models
$\mathcal M_1$ and $\mathcal M_2$ as follows:

\begin{dfn}\label{notation}
For $\ell\in\{1,2\}$, let
\begin{itemize}
\item[(a)]
$\mathcal N_\ell=\scc(\mathcal M_\ell)$,
\item[(b)]
$N_\ell=\scc(M_\ell)$,
\item[(c)]
$j_\ell=j_{M_\ell}$.
\end{itemize}
\end{dfn}


It is easy to observe the following:

\begin{itemize}
\item[(a)]
$P^{\mathcal N_\ell}_1$ consists of all ultrafilters~$D$ 
on $M_\ell$ such that $P^{\mathcal M_\ell}_1\in D$
(so for $\ell=1$ this means $\mathbb N\in D$), and
$
P^{\mathcal N_\ell}_1\setminus
\{j_\ell(n):n\in P^{\mathcal M_\ell}_1\}
$
consists of all such non-principal ultrafilters,
\item[(b)]
$P^{\mathcal N_\ell}_2$ consists of all ultrafilters~$D$ 
on $M_\ell$ such that $P^{\mathcal M_\ell}_2\in D$
(so for $\ell=1$ this means $\mathcal P(\mathbb N)\in D$), 
and $
P^{\mathcal N_\ell}_2\setminus
\{j_\ell(A):A\in P^{\mathcal M_\ell}_2\}
$
consists of all such non-principal ultrafilters.
\end{itemize}


Now we are going to construct a~specific sentence~$\psi$ 
which will be satisfied in $\mathcal N_1$ but not
in~$\mathcal N_2$. First we define two auxiliary 
formulas $\varphi_1$ and~$\varphi_2$.

\begin{dfn}\label{definition phi}
For $i\in\{1,2\}$, let $\varphi_i(x)$ be 
the following formula in~$\tau$:
$$ 
P_i(x)\wedge\forall y\,(P_i(y)\to F(x,y)=F(y,x)).
$$
\end{dfn}

Thus $\varphi_i(x)$ means that $x$~is in the center 
in a~sense. Actually, only $\varphi_2$ will be used 
to construct~$\psi$.

\begin{prop}\label{properties phi}
Assume $i,\ell\in\{1,2\}$.
For every $D\in N_\ell$,
$$
\mathcal N_\ell\vDash\varphi_i(D) 
\;\;\text{if and only if}\;\;
D\in\bigl\{j_\ell(a):a\in P^{\mathcal M_\ell}_i\bigr\}.
$$
\end{prop}

\begin{proof}
This follows from the four lemmas below.

\begin{lmm}\label{lemma 1}
If $D\notin P^{\mathcal N_\ell}_i$ then 
$\mathcal N_\ell\vDash\neg\,\varphi_i(D)$. 
\end{lmm}

\begin{proof}
By the first conjunct in~$\varphi_i$.
\end{proof}

\begin{lmm}\label{lemma 2}
If $D_1\in P^{\mathcal N_\ell}_i$ and 
$D_2=j_\ell(a)$ for some $a\in P^{\mathcal M_\ell}_i$, 
then 
$$
\mathcal N_\ell\vDash F(D_1,D_2)=F(D_2,D_1).
$$
\end{lmm}

\begin{proof}
We must check that 
$
F^{\mathcal N_\ell}(D_1,D_2)=
F^{\mathcal N_\ell}(D_2,D_1).
$ 
It suffices to show that, for any 
$A\subseteq P^{\mathcal M_\ell}_i$, 
the following equivalence holds: 
$$ 
A\in F^{\mathcal N_\ell}(D_1,D_2)
\;\Leftrightarrow\;
A\in F^{\mathcal N_\ell}(D_2,D_1).
$$

By Definition~\ref{ultrafilter extension}, we have
$$
A\in F^{\mathcal N_\ell}(D_1,D_2)
\;\Leftrightarrow\;
(\forall^{D_1}x_1)(\forall^{D_2}x_2)\,
F^{\mathcal M_\ell}(x_1,x_2)\in A. 
$$
But $D_2=j_\ell(a)$ for an $a\in P^{\mathcal M_\ell}_i$, 
i.e., $D_2$~is a~principal ultrafilter given by~$a$. 
Hence $\forall^{D_2}x_2$ is reduced by replacing 
the bounded occurrence of the variable~$x_2$ with~$a$ 
(as we have noted after 
Definition~\ref{ultrafilter quantifier}), 
whence we get 
$$
A\in F^{\mathcal N_\ell}(D_1,D_2)
\;\Leftrightarrow\;
(\forall^{D_1}x_1)\,
F^{\mathcal M_\ell}(x_1,a)\in A.
$$
Similarly we get
$$
A\in F^{\mathcal N_\ell}(D_2,D_1)
\;\Leftrightarrow\;
(\forall^{D_1}x_1)\,
F^{\mathcal M_\ell}(a,x_1)\in A.
$$
Since $a\in P^{\mathcal M_\ell}_i$, we have 
$F^{\mathcal M_\ell}(a,b)=F^{\mathcal M_\ell}(b,a)$ 
for all $b\in P^{\mathcal M_\ell}_i$
by Definition~\ref{model 1}(f)($\alpha$). And since 
$P^{\mathcal M_\ell}_i\in D_1$, the required 
equivalence follows. 
\end{proof}


\begin{lmm}\label{lemma 3}
If 
$
D_1\in 
P^{\mathcal N_1}_1\setminus
\{j_1(n):n\in P^{\mathcal M_1}_1\}
$,
then there exists $D_2\in P^{\mathcal N_1}_1$ such that
$$
F^{\mathcal N_1}(D_1,D_2)\ne F^{\mathcal N_1}(D_2,D_1).
$$
\end{lmm}

\begin{proof}
Actually we shall prove a bit stronger assertion:
if
$
D_1,D_2\in 
P^{\mathcal N_1}_1\setminus
\{j_1(n):n\in P^{\mathcal M_1}_1\}
$
are such that $D_1\ne D_2$, then 
$$ 
F^{\mathcal N_1}(D_1,D_2)\ne F^{\mathcal N_1}(D_2,D_1). 
$$

So assume that $D_1,D_2$ are distinct non-principal 
ultrafilters on~$M_1$ such that $\mathbb N\in D_1\cap D_2$. 
By $D_1\ne D_2$, there is $A_1\in\mathcal P(\mathbb N)$ such 
that $A_1\in D_1$ and $A_2=\mathbb N\setminus A_1\in D_2$.
Let 
\begin{align*}
B_1&=
\bigl\{F^{\mathcal M_1}(n_1,n_2):
n_1\in A_1
\wedge 
n_2\in A_2 
\wedge 
(n_1,n_2)\in R^{\mathcal M_1}_2
\bigr\},
\\
B_2&=
\bigl\{F^{\mathcal M_1}(n_1,n_2):
n_1\in A_1 
\wedge 
n_2\in A_2 
\wedge 
(n_2,n_1)\in R^{\mathcal M_1}_2
\bigr\}.
\end{align*}
Recall that 
$R^{\mathcal M_1}_2\cap(\mathbb N\times\mathbb N)$ is 
the usual order~$<$ on~$\mathbb N$, so the last 
conjuncts in the definition of $B_1$ and $B_2$ mean 
just $n_1<n_2$ and $n_2<n_1$, respectively.

Now our stronger assertion clearly follows 
from claims (a)--(c) below:
\begin{itemize}
\item[(a)]
$B_1\cap B_2=\emptyset$,
\item[(b)]
$B_1\in F^{\mathcal N_1}(D_1,D_2)$,
\item[(c)]
$B_2\in F^{\mathcal N_1}(D_2,D_1)$.
\end{itemize}
It remains to verify these claims.

For~(a), note that 
if there is some $c\in B_1\cap B_2$, then: 
\begin{itemize}
\item[($\alpha$)]
since $c\in B_1$, we can find $n_1<n_2$ such that 
$F^{\mathcal M_1}(n_1,n_2)=c$,
$n_1\in A_1$, 
$n_2\in A_2$, 
\item[($\beta$)]
since $c\in B_2$, we can find $m_2<m_1$ such that 
$F^{\mathcal M_1}(m_1,m_2)=c$,
$m_1\in A_1$, 
$m_2\in A_2$. 
\end{itemize}
So, since by Definition~\ref{model 1}(f)($\alpha$),
$F^{\mathcal M_1}$~is an unordered pairing function, 
we conclude $\{n_1,n_2\}=\{m_1,m_2\}$. However, then 
$n_1<n_2$ and $m_2<m_1$ imply $n_1=m_2$ and $n_2=m_1$, 
which contradicts to $n_1\in A_1$, $m_2\in A_2$.

For~(b), 
note that $\{n_2\in A_2:n_2>n_1\}\in D_2$ because of
$A_2\in D_2$ and $D_2$~is non-principal. It follows
$(\forall^{D_2}n_2)\,F(n_1,n_2)\in B_1$.
But $A_1\in D_1$, so we get
$$ 
(\forall^{D_1}n_1)(\forall^{D_2}n_2)\,
F(n_1,n_2)\in B_1.
$$
By Definition~\ref{ultrafilter extension}(d), 
this gives claim~(b).

For~(c), argue similarly.
\end{proof}


The fourth lemma (and its proof) generalizes 
the previous one.

\begin{lmm}\label{lemma 4}
If $i,\ell\in\{1,2\}$ and
$
D_1\in 
P^{\mathcal N_\ell}_i\setminus
\{j_\ell(a):a\in P^{\mathcal M_\ell}_i\},
$
then there exists $D_2\in P^{\mathcal N_\ell}_i$ 
such that
$$ 
F^{\mathcal N_\ell}(D_1,D_2)\ne 
F^{\mathcal N_\ell}(D_2,D_1). 
$$
\end{lmm}

\begin{proof}
Let $D_1$~be a~non-principal ultrafilter 
on~$P^{\mathcal M_\ell}_i$. It follows from 
Definition~\ref{model 1}(e) and 
$\mathcal M_1\prec\mathcal M_2$ that
$R^{\mathcal M_\ell}_2$~is a~linear order 
on~$P^{\mathcal M_\ell}_i$. 
One of the two following possibilities occurs: 
\begin{itemize}
\item[(a)]
there is an initial segment~$I$ of 
the linearly ordered set 
$(P^{\mathcal M_\ell}_i,R^{\mathcal M_\ell}_2)$ 
such that $I\in D_1$ but if $I_1\subset I$ 
is another initial segment of the set
then $I_1\notin D_1$ 
(this $I$ necessarily has no last element); 
\item[(b)]
there is a~final segment~$J$ of 
the linearly ordered set 
$(P^{\mathcal M_\ell}_i,R^{\mathcal M_\ell}_2)$ 
such that $J\in D_1$ but if $J_1\subset J$ 
is another final segment of the set
then $J_1\notin D_1$ 
(this $J$ necessarily has no first element). 
\end{itemize}

To see, notice the following general facts.
If $(X,<)$ is a~linearly ordered set, for 
any ultrafilter~$D$ on~$X$ define the initial 
segment $I_D$ and the final segment $J_D$ 
of $(X,<)$ as follows:
\begin{align*}
I_D
&=\,
\bigcap\,\{
I\in D:
I\text{ is an initial segment of }(X,<)\},
\\
J_D
&=\,
\bigcap\,\{
J\in D:
J\text{ is a~final segment of }(X,<)\}.
\end{align*}
As easy to see, if $D$~is principal then 
$I_D\cap J_D=\{x\}$ for $\{x\}\in D$; and 
if $D$~is non-principal then $(I_D,J_D)$ is a~cut 
and either $I_D$ or $J_D$, but not both, is in~$D$.
Furthermore, if $I_D$ is in~$D$, then so are 
all final segments of~$I_D$, $S\cap I_D$ is cofinal 
in~$I_D$ for all $S\in D$, and $I_D$~does not have 
a~greatest element whenever $D$~is non-principal; 
and symmetrically for $J_D$ in~$D$.
(More details related to ultrafilter extensions 
of linearly ordered sets can be found
in~\cite{Saveliev(orders)}.)

In our situation, $D_1$~is non-principal, so we have 
either $I_{D_1}\in D_1$, in which case we get 
possibility~(a) with $I=I_{D_1}$, or $J_{D_1}\in D_1$, 
in which case we get possibility~(b) with $J=J_{D_1}$.

For~(a), choose an ultrafilter~$D_2$ 
on~$P^{\mathcal M_\ell}_i$ such that 
\begin{itemize}
\item[($\alpha$)]
$I\in D_2$,
\item[($\beta$)]
if $I_1\subset I$ is an initial segment of 
$(P^{\mathcal M_\ell}_i,R^{\mathcal M_\ell}_2)$ 
then $I_1\notin D_2$, 
\item[($\gamma$)]
$D_2\ne D_1$.
\end{itemize}
Now we can repeat the proof of Lemma~\ref{lemma 3} 
mutatis mutandis, i.e., we can find
$A_1\in D_1\setminus D_2$ such that $A_1\subseteq I$ 
and $A_2=I\setminus A_1\in D_2$, then define
\begin{align*}
B_1&=
\bigl\{F^{\mathcal M_\ell}(a_1,a_2):
a_1\in A_1
\wedge 
a_2\in A_2
\wedge 
(a_1,a_2)\in R^{\mathcal M_\ell}_2 
\bigr\}, 
\\
B_2&=
\bigl\{F^{\mathcal M_\ell}(a_1,a_2):
a_1\in A_1
\wedge 
a_2\in A_2
\wedge 
(a_2,a_1)\in R^{\mathcal M_\ell}_2 
\bigr\},
\end{align*}
etc.

For~(b), the proof is symmetric: we only replace 
$I$ with~$J$, initial segments with final ones, 
and $xR^{\mathcal M_\ell}_2y$ 
with $yR^{\mathcal M_\ell}_2x$. 
\end{proof}

These four lemmas complete the proof of 
Proposition~\ref{properties phi}.
\end{proof}


Now everything is ready in order to provide 
a~sentence~$\psi$ having the required property.

\begin{dfn}\label{definition psi}
Let $\psi$ be the following sentence in~$\tau$: 
\begin{gather*}
(\forall x_1)
(\forall x_2)\,
\bigl(P_1(x_1)\wedge P_1(x_2)\wedge x_1\ne x_2
\qquad\qquad\qquad\qquad\qquad\qquad
\\
\qquad\qquad\qquad\qquad\qquad\qquad
\,\to\,
(\exists y)\,\varphi_2(y)\wedge
R_1(x_1,y)\wedge\neg\,R_1(x_2,y)
\bigr).
\end{gather*}
\end{dfn}


\begin{prop}\label{properties psi}
Let $\ell\in\{1,2\}$. Then 
\begin{align*}
\mathcal N_\ell\vDash\psi
\;\;\text{if and only if}\;\;
\ell=1.
\end{align*}
\end{prop}

\begin{proof}
1. 
First we show that $\mathcal N_1\vDash\psi$.

Let $D_1,D_2$ satisfy the antecedent of~$\psi$, i.e., 
$D_1,D_2\in P^{\mathcal N_1}_1$ and $D_1\ne D_2$. 
We should find $b\in N_1$ such that
$$
\mathcal N_1\vDash
\varphi_2(b)\wedge
R_1(D_1,b)\wedge\neg\,R_1(D_2,b).
$$

Since $D_1,D_2$ are distinct ultrafilters on~$M_1$ such that 
$P^{\mathcal M_1}_1\in D_1\cap D_2$, we can choose
$A_1\subseteq P^{\mathcal M_1}_1$ such that 
$A_1\in D_1$ and $A_1\notin D_2$. 
Then $A_1\in P^{\mathcal M_1}_2$ clearly follows 
from Definition~\ref{model 1}(b),(c). 
So $b=j_1(A_1)\in P^{\mathcal N_1}_2$, and hence, 
by the ``if" part of Proposition~\ref{properties phi},
$\mathcal N_1\vDash\varphi_2(b)$.

It remains to show the conjunction 
$$
(D_1,b)\in R^{\mathcal N_1}_1 
\;\;\text{and}\;\;
(D_2,b)\notin R^{\mathcal N_1}_1.
$$
To this end, note that for any ultrafilter~$D$ concentrated 
on $P^{\mathcal M_1}_1$ and any $A\in P^{\mathcal M_1}_2$, 
by Definition~\ref{ultrafilter extension}(c), the formula 
$(D,j_1(A))\in R^{\mathcal N_1}_1$ means
$$
(\forall^Dn)(\forall^{j(A)}B)\,
(n,B)\in R^{\mathcal M_1}_1.
$$
Recalling that $R^{\mathcal M_1}_1$~is the membership relation 
(Definition~\ref{model 1}(d)) and reducing~$(\forall^{j(A)}B)$, 
we see that the latter formula is equivalent to 
$(\forall^Dn)\,n\in A$, and so, to $A\in D$.
Since we have $A_1\in D_1$ and $A_1\notin D_2$, 
this gives the required conjunction.


2. 
Now we show that $\mathcal N_2\vDash\neg\,\psi$.

Define a~function~$G$ from $P^{\mathcal N_2}_1$ 
into $\mathcal P(P^{\mathcal M_2}_2)$ as follows:
$$
G(D)=
\bigl\{b\in P^{\mathcal M_2}_2:
\bigl\{a\in P^{\mathcal M_2}_1:
(a,b)\in R^{\mathcal M_2}_1\bigr\}\in D
\bigr\}.
$$
Recall that 
$|P^{\mathcal M_2}_1|=|P^{\mathcal M_2}_1|=\lambda$
(Proposition~\ref{model 2}). Therefore,
$$
|\dom(G)|=
|\scc(|P^{\mathcal M_2}_1|)=
|\scc(\lambda)|=
2^{2^\lambda}>2^\lambda,
$$
while
$$
|\ran(G)|\le
|\mathcal P(P^{\mathcal M_2}_2)|=
|\mathcal P(\lambda)|=2^\lambda,
$$
whence we conclude that $G$~is not one-to-one.

Take $S\in\mathcal P(P^{\mathcal M_2}_2)$ such that 
$|G^{-1}(S)|>1$, pick $D_1,D_2\in G^{-1}(S)$ such that 
$D_1\ne D_2$, and show that $D_1,D_2$ witness the failure 
of the sentence~$\psi$.

Note that $\mathcal N_2$ satisfies the antecedent of~$\psi$, 
i.e.,
$$
\mathcal N_2\vDash
P_1(D_1)\wedge P_1(D_2)\wedge D_1\ne D_2,
$$
by the condition 
$D_1,D_2\in G^{-1}(S)\subseteq P^{\mathcal N_2}_1$.
So to finish, it suffices to show 
$$
\mathcal N_2\vDash
\neg\,(\exists y)\,
\varphi_2(y)\wedge R_1(D_1,y)\wedge\neg\,R_1(D_2,y).
$$

Toward a~contradiction, assume that there is $b\in N_2$ 
such that
$$
\mathcal N_2\vDash
\varphi_2(b)\wedge R_1(D_1,b)\wedge\neg\,R_1(D_2,b).
$$
But since $\mathcal N_2\vDash\varphi_2(b)$, 
by the ``only if" part of Proposition~\ref{properties phi},
we see that $b=j_2(A)$ for some $A\in P^{\mathcal M_2}_2$. 
So we obtain
$$
R^{\mathcal N_2}_1(D_1,j_2(A)) 
\;\;\text{and}\;\;
\neg\,R^{\mathcal N_2}_1(D_2,j_2(A))
$$

By Definition~\ref{ultrafilter extension}(c), 
$R^{\mathcal N_2}_1(D_1,j_2(A))$ means
$
(\forall^{D_1}a)(\forall^{j_2(A)}b)\,
(a,b)\in R^{\mathcal M_2}_1,
$
whence reducing~$(\forall^{j_2(A)}b)$ we get 
$(\forall^{D_1}a)\,(a,A)\in R^{\mathcal M_2}_1$,
i.e., 
$$
\bigl\{a\in P^{\mathcal M_2}_1:
(a,A)\in R^{\mathcal M_2}_1\bigr\}
\in D_1.
$$
Similarly,
$R^{\mathcal N_2}_1(D_2,j_2(A))$ is equivalent to
$
\{a\in P^{\mathcal M_2}_1:
(a,A)\in R^{\mathcal M_2}_1\}
\in D_2,
$
and hence,
$\neg\,R^{\mathcal N_2}_1(D_2,j_2(A))$ is equivalent to
$$
\bigl\{a\in P^{\mathcal M_2}_1:
(a,A)\in R^{\mathcal M_2}_1\bigr\}
\notin D_2.
$$
Therefore, $A\in G(D_1)$ and $A\notin G(D_2)$, which,
however, contradicts to the choice of $D_1,D_2$.

This completes the proof.
\end{proof}

So we have constructed two models 
$\mathcal M_1$, $\mathcal M_2$ in~$\tau$ with 
$$
\mathcal M_1\prec\mathcal M_2
$$ 
and 
a~$\tau$-sentence~$\psi$ such that 
$\mathcal N_1=\scc(\mathcal M_1)\vDash\psi$ and 
$\mathcal N_2=\scc(\mathcal M_2)\vDash\neg\,\psi$,
thus witnessing 
$$
\scc(\mathcal M_1)\not\equiv\scc(\mathcal M_2).
$$ 
This proves the Main Theorem (Theorem~\ref{main theorem}).


%
%
%


\vskip+2em

\begin{footnotesize} 

\noindent
{\sc
The Russian Academy of Sciences, 
Institute for Information Transmission Problems, 
Bolshoy Karetny per.~19, build.~1, 
Moscow 127051 Russia
\/}
\\
{\it E-mail address:\/} 
d.i.saveliev@iitp.ru, 
d.i.saveliev@gmail.com

\vskip+0.5em

\noindent
{\sc
Einstein Institute of Mathematics, 
The Hebrew University of Jerusalem, 
\\Jerusalem 9190401 Israel, 
and 
Rutgers, 
The State University of New Jersey, 
110 Frelinghuysen Road, 
Piscataway, NJ 08854-8019 USA
\/}
\\
{\it E-mail address:\/} 
shelah@math.huji.ac.il

\end{footnotesize}

\end{document}